\documentclass[11pt]{article}

\usepackage[margin=1in]{geometry}
\usepackage[T1]{fontenc}
\usepackage[utf8]{inputenc}
\usepackage{lmodern}
\usepackage{amsmath,amssymb,amsthm,mathtools}
\usepackage{microtype}
\usepackage{xcolor}
\usepackage{hyperref}
\usepackage[nameinlink,noabbrev]{cleveref}

\date{}

\theoremstyle{plain}
      \newtheorem{theorem}{Theorem}[section]
      \newtheorem{lemma}[theorem]{Lemma}
      \newtheorem{claim}[theorem]{Claim}

      \newtheorem{conjecture}[theorem]{Conjecture}
\theoremstyle{definition}

\theoremstyle{remark}

\title{A parity Erd\H{o}s--Hajnal theorem for \(t\)-intersecting curves}

\begin{document}

\author{Andrew Suk\thanks{Department of Mathematics, University of California at San Diego, La Jolla, CA, 92093 USA. Supported by NSF grant DMS-2246847. Email: {\tt asuk@ucsd.edu}.}  \and Su Zhou\thanks{Department of Mathematics, University of California at San Diego, La Jolla, CA, 92093 USA. Supported by NSF grant DMS-2246847. Email: {\tt suzhou@ucsd.edu}.} }

\maketitle

\begin{abstract}
For every fixed \(t\ge 1\), we prove a parity analogue of the mighty
Erd\H{o}s--Hajnal property for \(t\)-intersecting curves in the plane. Let
\(\mathcal B\) be a set of blue curves and \(\mathcal G\) a set of green
curves in the plane such that \(\mathcal B\cup\mathcal G\) is a collection
of \(t\)-intersecting curves in general position. We show that there exist
subfamilies \(\mathcal B'\subseteq\mathcal B\) and
\(\mathcal G'\subseteq\mathcal G\) such that
\[
|\mathcal B'|\geq \varepsilon|\mathcal B|
\qquad\text{and}\qquad
|\mathcal G'|\geq \varepsilon|\mathcal G|,
\]
where \(\varepsilon>0\) depends only on \(t\), such that either every pair in
\(\mathcal B'\times\mathcal G'\) intersects an even number of times or every
such pair intersects an odd number of times. For \(t=1\), this recovers the
theorem of Fox, Pach, and Suk for pseudo-segments. As an application, we show
that every \(n\)-vertex topological graph with edges forming a
\(t\)-intersecting family and with no \(k\) edges that pairwise cross an odd
number of times has at most \(n(\log n)^{O_t(\log k)}\) edges.
\end{abstract}

\section{Introduction}

Let \(\mathcal{C}\) be a finite family of simple Jordan arcs in the plane.
Throughout the paper, we assume that \(\mathcal{C}\) is in \emph{general
position}, that is, every two curves intersect in finitely many points, no two
curves are tangent, no three curves pass through the same point, and no
endpoint of one curve lies on another curve. In particular, every intersection
between two curves is a proper crossing. For an integer \(t\ge 1\), we say
that \(\mathcal{C}\) is \emph{\(t\)-intersecting} if every pair of curves in
\(\mathcal{C}\) intersect in at most \(t\) points. In the case \(t=1\), we
refer to \(\mathcal{C}\) as a collection of \emph{pseudo-segments}. The
\emph{intersection graph} of \(\mathcal C\) is the graph whose vertex set is
\(\mathcal C\), where two vertices are adjacent if and only if the
corresponding curves intersect. A graph is called a \emph{string graph} if it
arises as the intersection graph of curves in the plane.

The celebrated Erd\H{o}s--Hajnal conjecture states that every hereditary
graph class defined by forbidding a fixed induced subgraph contains a clique
or an independent set of polynomial size. We say that a hereditary graph class has the \emph{Erd\H{o}s--Hajnal
property} if there is a constant \(\delta>0\) such that every \(n\)-vertex
graph in the class contains a clique or an independent set of size at least
\(n^\delta\).  A breakthrough result of Tomon~\cite{Tomon} shows that the class of string graphs has this property.

One way to prove that a hereditary graph class has the
Erd\H{o}s--Hajnal property is to establish a stronger bipartite statement.
A hereditary graph class is said to have the \emph{strong
Erd\H{o}s--Hajnal property} if there exists a constant \(\delta>0\) such
that every graph \(G\) in the class contains two
disjoint subsets \(A,B\subseteq V(G)\), each of size at least
\(\delta|V(G)|\), such that the bipartite graph between \(A\) and \(B\) is
either complete or empty. It is known that string graphs do not have the
strong Erd\H{o}s--Hajnal property~\cite{Fox06,PT06}. On the other hand, Fox,
Pach, and T\'oth~\cite{FoxPachToth} proved that, for every fixed \(t\ge 1\),
intersection graphs of \(t\)-intersecting curves have the strong
Erd\H{o}s--Hajnal property.

For some applications, including density and strong regularity-type results, even this bipartite strengthening is not enough.  A
hereditary graph class is said to have the \emph{mighty
Erd\H{o}s--Hajnal property}\footnote{Fox, Pach, and Suk~\cite{FPS} define this property assuming
\(|A|=|B|\). We use the version without this assumption. For string graphs, the two versions are equivalent by a cloning argument.} if there
exists a constant \(\delta>0\) such that, for every graph \(G\) in the class
and every pair of disjoint subsets \(A,B\subseteq V(G)\), there exist
subsets \(A'\subseteq A\) and \(B'\subseteq B\) satisfying
\[
|A'|\ge \delta|A|
\qquad\text{and}\qquad
|B'|\ge \delta|B|,
\]
such that the bipartite graph between \(A'\) and \(B'\) is either complete
or empty.  Pach and
Solymosi~\cite{PachSolymosi} proved that intersection graphs of segments
in the plane have the mighty Erd\H{o}s--Hajnal property, and more recently,
Fox, Pach, and Suk~\cite{FPS} extended this to pseudo-segments via the following theorem.

\begin{theorem}[Fox--Pach--Suk~\cite{FPS}]
There exists a constant \(\varepsilon>0\) with the following property. Let
\(\mathcal B\) be a set of blue curves, and let \(\mathcal G\) be a set of
green curves in the plane such that \(\mathcal B\cup \mathcal G\) is a
collection of pseudo-segments in general position. Then there exist subsets
\(\mathcal B'\subseteq \mathcal B\) and \(\mathcal G'\subseteq \mathcal G\),
with
\[
|\mathcal B'|\ge \varepsilon|\mathcal B|
\qquad\text{and}\qquad
|\mathcal G'|\ge \varepsilon|\mathcal G|,
\]
such that either every curve in \(\mathcal B'\) intersects every curve in
\(\mathcal G'\), or every such pair is disjoint.
\end{theorem}

It is not known whether intersection graphs of \(t\)-intersecting curves
satisfy the mighty Erd\H{o}s--Hajnal property for \(t\ge 2\). Kor\'andi,
Pach, and Tomon~\cite{KPT} conjectured that they do and proved a partial
result in a double-grounded setting. More recently, Hunter, Milojevi\'c,
Sudakov, and Tomon~\cite{HMT} proved that if the bipartite intersection graph
between two \(n\)-element \(t\)-intersecting families has positive density,
then it contains a linear-size complete bipartite graph. A similar density
result was also obtained by Fox, Pach, and T\'oth~\cite{FoxPachToth}.

Our main result establishes a parity analogue of the mighty Erd\H{o}s--Hajnal
property for \(t\)-intersecting curves. For \(t=1\), this recovers the theorem of Fox, Pach, and Suk for
pseudo-segments. 

\begin{theorem}\label{main}
For every integer \(t \ge 1\), there exists a constant \(\varepsilon>0\) with
the following property. Let \(\mathcal B\) be a set of blue curves, and let
\(\mathcal G\) be a set of green curves in the plane such that
\(\mathcal B\cup \mathcal G\) is a collection of \(t\)-intersecting curves in
general position. Then there exist subsets
\(\mathcal B'\subseteq \mathcal B\) and
\(\mathcal G'\subseteq \mathcal G\), with
\[
|\mathcal B'|\ge \varepsilon|\mathcal B|
\qquad\text{and}\qquad
|\mathcal G'|\ge \varepsilon|\mathcal G|,
\]
such that either every curve in \(\mathcal B'\) intersects every curve in
\(\mathcal G'\) an odd number of times, or every such pair intersects an even
number of times.
\end{theorem}

Fox, Pach, and Suk~\cite[Theorem~2.2]{FPS} gave several equivalent
formulations of the mighty Erd\H{o}s--Hajnal property. One of them, together
with Theorem~\ref{main}, yields the following strengthening of
Szemer\'edi's regularity lemma for \(t\)-intersecting curves.

\begin{theorem} For every integer \(t\ge 1\) and every \(\varepsilon>0\), there exists \(K=K(t,\varepsilon)\) such that the following holds. Let \(\mathcal{C}\) be a finite collection of \(t\)-intersecting curves in the plane in general position. Then there is an equipartition of \(\mathcal{C}\) into \(K\) parts \(\mathcal{C}_1,\dots,\mathcal{C}_K\) such that, for all but at most \(\varepsilon K^2\) pairs \(\mathcal{C}_i,\mathcal{C}_j\) of parts, either every curve in \(\mathcal{C}_i\) intersects every curve in \(\mathcal{C}_j\) an even number of times, or every curve in \(\mathcal{C}_i\) intersects every curve in \(\mathcal{C}_j\) an odd number of times. \end{theorem}

\noindent A standard application of a weak bipartite regularity lemma of
Koml\'os and S\'os~\cite[Theorem~9.4.1]{Matousek}, together with
Theorem~\ref{main}, yields the following density result, which will be used in Section~\ref{sec:top}.

\begin{theorem}\label{thm:dense}
For every integer \(t \ge 1\), there exists a constant \(c=c(t)>0\) with the
following property. Let \(\mathcal B\) be a set of blue curves, and let
\(\mathcal G\) be a set of green curves in the plane such that
\(\mathcal B\cup \mathcal G\) is a collection of \(t\)-intersecting curves in
general position. Suppose that at least $\delta |\mathcal B||\mathcal G|$ pairs \((\beta,\gamma)\in \mathcal B\times \mathcal G\) intersect an odd
number of times. Then there exist subsets
\(\mathcal B'\subseteq \mathcal B\) and
\(\mathcal G'\subseteq \mathcal G\), with
\[
|\mathcal B'|\ge \delta^c |\mathcal B|
\qquad\text{and}\qquad
|\mathcal G'|\ge \delta^c |\mathcal G|,
\]
such that every curve in \(\mathcal B'\) intersects every curve in
\(\mathcal G'\) an odd number of times.
\end{theorem}

\noindent Likewise, an analogue of Theorem~\ref{thm:dense} holds for even crossings.

\medskip

\noindent \textbf{Quasi-planar graphs and odd crossings.}  A \emph{topological graph} is a graph drawn in the plane with points as vertices and simple Jordan arcs as edges connecting the corresponding endpoints. All topological graphs considered in this paper have no loops or multiple edges.  The edges are allowed to cross, but may not pass through vertices other than their endpoints. We always assume the drawing is in general position: no two edges are tangent,
no three edges pass through the same interior point, and every intersection
between the interiors of two edges is a proper crossing.  We say that the edges of a topological graph are \emph{\(t\)-intersecting} if no two edges cross more than \(t\) times.

Euler's formula implies that every planar graph on \(n\) vertices has at most
\(3n-6\) edges. A longstanding open problem in topological graph theory asks
whether a similar linear bound holds under the much weaker assumption that no
\(k\) edges pairwise cross. A topological graph is called
\emph{\(k\)-quasi-planar} if it contains no \(k\) pairwise crossing edges.

\begin{conjecture}
For every integer \(k\ge 2\), every \(k\)-quasi-planar topological graph on
\(n\) vertices has \(O_k(n)\) edges.
\end{conjecture}

Pach, Radoi\v{c}i\'c, and T\'oth~\cite{PRT04} proved the conjecture for
\(k=3\), and Ackerman and Tardos~\cite{AcT07} later improved the constant.
Ackerman~\cite{Ac09} proved the conjecture for \(k=4\). The conjecture remains open for
every \(k\ge 5\). For general \(k\), the best known bound is due to Fox,
Pach, and Suk~\cite{FPS2}, who proved that every \(k\)-quasi-planar
topological graph on \(n\) vertices has at most
\(n(\log n)^{O(\log k)}\) edges.

By the celebrated Hanani--Tutte theorem (see~\cite{ThomassenHT}), if an $n$-vertex topological graph has no two edges cross an odd number of times, then the graph is planar and has \(O(n)\) edges. In~\cite{PTdisjoint}, Pach and T\'oth established the following parity analogue of the quasi-planar graph problem.

\begin{theorem}[Pach--T\'oth~\cite{PTdisjoint}]
For every \(k \ge 2\), there exists a constant \(C=C(k)\) such that the following holds. If \(G\) is an \(n\)-vertex topological graph with no \(k\)  edges that pairwise cross an odd number of times, then
\[
|E(G)| \le Cn(\log n)^{5k-10}.
\]
\end{theorem}

It is possible that, for every fixed \(k\), every \(n\)-vertex topological graph with no \(k\) edges that pairwise cross an odd number of times has only \(O_k(n)\) edges. While this remains open for all $k\geq 3$, our next result gives a substantially stronger bound for large \(k\), under the additional assumption that the edges of \(G\) form a \(t\)-intersecting family. 

\begin{theorem}\label{oddcross}
For all integers \(k\ge 2\) and $t\geq 1$, there exists a constant \(C=C(t)>0\) such that the following holds. Let \(G\) be an \(n\)-vertex topological graph whose edges form a \(t\)-intersecting family. If \(G\) contains no \(k\) edges that pairwise cross an odd number of times, then
\[
|E(G)| \le n(\log n)^{C\log k}.
\]
\end{theorem}

Our paper is organized as follows.  In the next section, we recall
\(r\)-divisions and modify the construction of Klein, Mozes, and Sommer
\cite{KMS13} to obtain a version with prescribed vertex constraints.  In
Section~\ref{sec:intersecting}, we use this \(r\)-division theorem to prove
a structural result for arrangements of blue and green curves, from which
Theorem~\ref{main} follows.  In Section~\ref{sec:top}, we apply
Theorem~\ref{thm:dense} to topological graphs and prove Theorem~\ref{oddcross}.

\section{An $r$-division with vertex constraints}

In this section we prove a variant of the $r$-division theorem that will be used later. The notion of an $r$-division was introduced by Frederickson~\cite{Fre87} as a refinement of the planar separator method of Lipton and Tarjan~\cite{LiptonTarjan1979}. A key ingredient in the construction is the following simple-cycle separator theorem of Miller.

\begin{theorem}[Miller~\cite{Miller1986}]\label{cycle}
Let $G$ be a triangulated planar graph with $N$ vertices embedded in the plane, and let $w:V(G)\to \mathbb R_{\ge 0}$ be a weight function with total weight $W$. Then there exists a simple cycle $C$ in $G$ such that $|C|=O(\sqrt N)$, and the total weight of the vertices strictly inside $C$ and strictly outside $C$ are both at most $2W/3$.
\end{theorem}

Let $G=(V,E,F)$ be a planar graph embedded in the plane. A \emph{region} of $G$ is a connected edge-induced subgraph of $G$. A \emph{division} of $G$ is a collection of regions $\mathcal R=\{R_1,\ldots,R_k\}$ such that every edge of $G$ lies in at least one region. Given a division $\mathcal R$ of $G$, a vertex $v\in V(G)$ is a \emph{boundary vertex} if it belongs to more than one region in $\mathcal R$. For a region $R\in \mathcal R$, we denote by $b(R)$ the set of vertices of $R$ that are boundary vertices. A \emph{hole} of a region \(R\) is a face of \(R\) that is not a face of \(G\).

\begin{theorem}\label{thm:mixed}
There exists an absolute constant \(r_0>0\) such that the following holds. Let \(G\) be a triangulated planar graph on \(N\) vertices embedded in the plane, and let \(Q\subseteq V(G)\). Then for every \(r\ge r_0\) and every \(\rho\ge 1\), there exists a division \(\mathcal R\) of \(G\) such that
\[
|\mathcal R|=O\!\left(\frac{N}{r}+\frac{|Q|}{\rho}\right),
\]
and every region \(R\in\mathcal R\) satisfies
\[
|V(R)|=O(r),\qquad |b(R)|=O(\sqrt r),\qquad |(V(R)\setminus b(R))\cap Q|=O(\rho),
\]
and has at most \(24\) holes.
\end{theorem}

\noindent The proof of Theorem~\ref{thm:mixed} is a simple modification of the recursive
$r$-division construction of Klein, Mozes, and Sommer~\cite{KMS13}. The only
change is that we add one more balancing step, in order to control the set \(Q\).
For the reader's convenience, we include the details.

We shall use the following recurrence estimate from Klein, Mozes, and Sommer~\cite[Section~2.7]{KMS13}. We state the form needed here, which follows from their proof with only minor changes in the constants.

\begin{lemma}\label{lem:kms-recurrence}
Let \(1/2\le \lambda<1\), and let \(c,c'\) be positive constants. Suppose that, for every \(\tau\ge 1\), a function \(T_\tau\) satisfies
\[
T_\tau(m)\le
\begin{cases}
c' m^\lambda+\displaystyle\max_{\alpha_1,\ldots,\alpha_{16}}
       \sum_{i=1}^{16} T_\tau(\alpha_i m),& m>\tau,\\[1ex]
0,&m\le \tau,
\end{cases}
\]
where the maximum is over all \(\alpha_1,\ldots,\alpha_{16}\ge 0\) such that
\[
\alpha_i\le {2\over 3}+{c\over \sqrt m}\quad\text{for all }i,
\qquad
1\le \sum_{i=1}^{16}\alpha_i\le 1+{c\over \sqrt m}.
\]
Then, for \(\tau\) sufficiently large,
\[
T_\tau(m)=O\!\left({m\over \tau^{1-\lambda}}\right),
\]
where the implicit constant depends only on \(c,c'\) and \(\lambda\).
\end{lemma}

\begin{proof}[Proof of Theorem \ref{thm:mixed}]
Let \(G\) be a triangulated planar graph on \(N\) vertices embedded in the plane, and let \(Q\subseteq V(G)\). We construct a rooted binary tree \(\mathcal T\) as follows. Each node \(x \in \mathcal T\) corresponds to a connected embedded region \(R_x\subseteq G\), with the root \(\hat x \in \mathcal T\) corresponding to \(G\). For each node \(x\), write
\[
n(x):=|V(R_x)|,\qquad
\beta(x):=|b(R_x)|,\qquad
q(x):=|(V(R_x)\setminus b(R_x))\cap Q|,
\]
and let \(h(x)\) denote the number of holes of \(R_x\).

Fix a sufficiently large absolute constant \(c_0\). The leaves of the final decomposition will be the rootmost nodes \(x\) satisfying
\begin{equation}\label{eq:leaf-threshold-rdiv}
n(x)\le c_0 r,\qquad
\beta(x)\le c_0\sqrt r,\qquad
q(x)\le c_0\rho.
\end{equation}
Moreover, we shall show that every region created in the recursion has at most \(24\) holes.

Suppose that \(x\in \mathcal T\) is not a leaf, and let \(R_x\) be the corresponding region. We construct the two children of \(R_x\) as follows. Since \(G\) is triangulated, every natural face of \(R_x\) is already a triangle. We triangulate each hole of \(R_x\) by placing one \emph{artificial vertex} in the hole and joining it to all vertices on the boundary of that hole with \emph{artificial edges}. This gives a triangulated planar graph \(R'_x\). Since the number of holes will be bounded by an absolute constant, \(R'_x\) has \(O(n(x))\) vertices.

Let \(d(x)\) be the depth of \(x\) in \(\mathcal T\). We choose the weight function in \Cref{cycle} according to \(d(x) \text{ } (\bmod \text{ } 4 )\). If \(d(x)\equiv 0\pmod 4\), then every vertex of \(R_x\) has weight \(1\). If \(d(x)\equiv 1\pmod 4\), then precisely the vertices in \(b(R_x)\) have weight \(1\). If \(d(x)\equiv 2\pmod 4\), then precisely the vertices in
\[
(V(R_x)\setminus b(R_x))\cap Q
\]
have weight \(1\). Finally, if \(d(x)\equiv 3\pmod 4\), then precisely the artificial vertices corresponding to the holes of \(R_x\) have weight \(1\). In all cases, every other vertex of \(R'_x\) has weight \(0\). Thus the separator balances, in turn, the number of vertices, the number of boundary vertices, the number of non-boundary vertices of \(Q\), and the number of holes.

Let \(C_x\) be the separator cycle. Then $|C_x|\le c_1\sqrt{n(x)}$, for an absolute constant \(c_1\). Let \(R_{x,0}\) and \(R_{x,1}\) be the two subregions consisting of the edges of \(R_x\) inside-or-on and outside-or-on \(C_x\), after deleting the artificial vertices and artificial edges. These two regions are the children of \(R_x\).

If \(x_0,x_1\) are the children of \(x\), then
\begin{equation}\label{eq:sum-n-beta-q}
n(x_0)+n(x_1)\le n(x)+c_1\sqrt{n(x)},
\end{equation}
\begin{equation}\label{eq:sum-beta}
\beta(x_0)+\beta(x_1)\le \beta(x)+c_1\sqrt{n(x)},
\end{equation}
and
\begin{equation}\label{eq:sum-q}
q(x_0)+q(x_1)\le q(x).
\end{equation}
Indeed, the only vertices that can be duplicated are vertices on \(C_x\). If a vertex of \(Q\) lies on \(C_x\), then it becomes a boundary vertex in the children, and so it is not counted by \(q\) in either child.

The balancing step gives the following additional estimates. If \(d(x)\equiv 0\pmod 4\), then
\begin{equation}\label{eq:n-balanced}
\max\{n(x_0),n(x_1)\}\le {2\over 3}n(x)+c_1\sqrt{n(x)}.
\end{equation}
If \(d(x)\equiv 1\pmod 4\), then
\begin{equation}\label{eq:beta-balanced}
\max\{\beta(x_0),\beta(x_1)\}\le {2\over 3}\beta(x)+c_1\sqrt{n(x)}.
\end{equation}
If \(d(x)\equiv 2\pmod 4\), then
\begin{equation}\label{eq:q-balanced}
\max\{q(x_0),q(x_1)\}\le {2\over 3}q(x).
\end{equation}

We next bound the number of holes. Say that a hole of \(R_x\) is fully contained on one side of \(C_x\) if all triangles obtained from its artificial triangulation lie on that side. Suppose that \(k\) holes are fully contained on one side of \(C_x\). Then the corresponding child has at most \(k+1\) holes, as the separator creates at most one new hole. Thus a non-hole-balancing step increases the number of holes along a branch by at most one. At a hole-balancing step, each child receives at most two-thirds of the old holes, plus one new hole. Hence, over one full four-level cycle,
\[
h(y)\le {2\over 3}h(x)+4.
\]
Thus, if \(h(x)\le 24\), then after one full four-level cycle
\[
h(y)\le {2\over 3}\cdot 24+4=20,
\]
and during the next three intermediate levels the number of holes is at most \(23\). Since the root has no holes, every region created by the recursion has at most \(24\) holes.

\medskip

We now bound the number of leaves of \(\mathcal T\). Since the leaves are precisely the regions in \(\mathcal R\), it remains to prove
\[
|\mathcal R|=O\!\left(\frac{N}{r}+\frac{|Q|}{\rho}\right).
\]
For a node \(x\), let \(S_r(x)\) be the set of rootmost descendants \(y\) of \(x\) such that $n(y)\le c_0r,$ but the parent of \(y\) has more than \(c_0r\) vertices. We say that a set $S$ of descendants of \(x\) is an \emph{antichain} if no one of them is an ancestor of another. For such an antichain \(S\), we define
\[
L(x,S):=-n(x)+\sum_{y\in S} n(y).
\]
Thus \(L(x,S)\) is the total excess in vertex multiplicity when
\(R_x\) is replaced by the regions \(R_y\), \(y\in S\). We shall use it exactly
as in the proof of Lemma~8 of Klein, Mozes, and Sommer~\cite{KMS13}.

Consider four consecutive levels of the recursion, starting at a node \(x\).
Let \(y_1,\ldots,y_k\) be the rootmost descendants of \(x\) which either have
already stopped with \(n(y_i)\le c_0r\), or have depth \(d(x)+4\). Then
\(k\le 16\), and the regions \(R_{y_1},\ldots,R_{y_k}\) cover \(R_x\).
Along each path from \(x\) to one of these descendants, the separator balances
the number of vertices once. By \eqref{eq:sum-n-beta-q} and \eqref{eq:n-balanced},
applied over this four-level block, we have
\[
\sum_{i=1}^k n(y_i)\le n(x)+c_2\sqrt{n(x)},
\qquad
\max_i n(y_i)\le {2\over 3}n(x)+c_2\sqrt{n(x)}
\]
for an absolute constant \(c_2\). Therefore, by the same recurrence argument as
in Lemma~8 of~\cite{KMS13}, the function
\[
B_r(m):=\max L(x,S_r(x)),
\]
where the maximum is over all nodes \(x\) with \(n(x)\le m\), satisfies the
recurrence in \Cref{lem:kms-recurrence} with \(\lambda=1/2\) and \(\tau=c_0r\). Hence
\begin{equation}\label{eq:first-L-bound}
L(\hat x,S_r(\hat x))=O\!\left({N\over \sqrt r}\right).
\end{equation}
It follows that
\[
\sum_{y\in S_r(\hat x)} n(y)
=
N+L(\hat x,S_r(\hat x))
=O(N).
\]
Since the parent of each \(y\in S_r(\hat x)\) has more than \(c_0r\) vertices, we get
\begin{equation}\label{eq:Sr-size}
|S_r(\hat x)|=O(N/r).
\end{equation}
Indeed, if a vertex belongs to \(t\) regions in this antichain, then it contributes \(t-1\) to \(L(\hat x,S_r(\hat x))\). The vertices with \(t\ge 2\) are precisely the boundary vertices of the antichain, and their total boundary multiplicity is at most twice their total contribution to \(L(\hat x,S_r(\hat x))\). Thus \eqref{eq:first-L-bound} gives
\begin{equation}\label{eq:boundary-sum}
\sum_{y\in S_r(\hat x)} \beta(y)
=
O\!\left({N\over \sqrt r}\right).
\end{equation}

Now fix \(x\in S_r(\hat x)\). All descendants of \(x\) have \(O(r)\) vertices and at most \(24\) holes. Let \(S'_r(x)\) be the set of rootmost descendants \(y\) of \(x\) such that
\[
\beta(y)\le c_0\sqrt r
\qquad\text{and}\qquad
q(y)\le c_0\rho.
\]
We claim that
\begin{equation}\label{eq:local-refinement}
|S'_r(x)|
=
O\!\left(1+{\beta(x)\over \sqrt r}+{q(x)\over \rho}\right).
\end{equation}

To prove the claim, again group the recursion below \(x\) into blocks of four levels. Let \(y_1,\ldots,y_k\) be the rootmost descendants of \(x\) that either satisfy the two stopping inequalities or have depth \(d(x)+4\). Then \(k\le 16\).  Since \(n(x)\le c_0r\), every region occurring in this four-level block has \(O(c_0r)\) vertices. Hence there is an absolute constant \(c_3\) such that
\[
\sum_{i=1}^k \beta(y_i)\le \beta(x)+c_3\sqrt{c_0r},
\qquad
\sum_{i=1}^k q(y_i)\le q(x).
\]
Moreover, for every \(y_i\) with \(|S'_r(y_i)|>1\), we have $\beta(y_i)\le {2\over3}\beta(x)+c_3\sqrt{c_0r}$ and $q(y_i)\le {2\over3}q(x).$  Choose \(c_0\) sufficiently large so that $c_0\ge 50$ and ${\sqrt{c_0}\over c_3}\ge 50.$  We prove by induction that
\[
|S'_r(x)|
\le
\max\left\{
1,\,
{\beta(x)\over c_3\sqrt{c_0r}}
+{q(x)\over\rho}
-15
\right\}.
\]
If \(x\) satisfies the two stopping inequalities, then \(S'_r(x)=\{x\}\). Otherwise, ${\beta(x)\over c_3\sqrt{c_0r}}+{q(x)\over\rho}>50.$  Let \(s\) be the cardinality of $\{i:|S'_r(y_i)|>1\}.$  We consider the following cases.

\medskip

\noindent \emph{Case 0.}  If \(s=0\), then
\[
|S'_r(x)|\le 16
\le
{\beta(x)\over c_3\sqrt{c_0r}}+{q(x)\over\rho}-15.
\]

\medskip

\noindent \emph{Case 1.} If \(s=1\), say \(|S'_r(y_1)|>1\), then
\[
\begin{aligned}
|S'_r(x)|
&\le |S'_r(y_1)|+k-1\\
&\le {\beta(y_1)\over c_3\sqrt{c_0r}}
     +{q(y_1)\over\rho}\\
&\le {2\over3}\left(
{\beta(x)\over c_3\sqrt{c_0r}}+{q(x)\over\rho}
\right)+1\\
&\le {\beta(x)\over c_3\sqrt{c_0r}}
     +{q(x)\over\rho}-15.
\end{aligned}
\]

\medskip

\noindent \emph{Case 2.} If \(s\ge 2\), let $I=\{i:|S'_r(y_i)|>1\}.$  Then
\[
\begin{aligned}
|S'_r(x)|
&\le
\sum_{i\in I}\left(
{\beta(y_i)\over c_3\sqrt{c_0r}}
+{q(y_i)\over\rho}-15
\right)+(k-s)\\
&\le
{\beta(x)\over c_3\sqrt{c_0r}}
+{q(x)\over\rho}
+1-15s+16-s\\
&\le
{\beta(x)\over c_3\sqrt{c_0r}}
+{q(x)\over\rho}-15.
\end{aligned}
\]
Hence
\[
|S'_r(x)|
=
O\!\left(
1+{\beta(x)\over\sqrt r}+{q(x)\over\rho}
\right).
\]

The leaves of the final decomposition are precisely the nodes in \(S'_r(x)\), over all \(x\in S_r(\hat x)\). Therefore, by \eqref{eq:Sr-size}, \eqref{eq:boundary-sum}, and the fact that non-boundary vertices of \(Q\) are counted in at most one region of any antichain,
\begin{align*}
|\mathcal R|
&\le \sum_{x\in S_r(\hat x)} |S'_r(x)| \\
&=
O\!\left(
|S_r(\hat x)|
+{1\over \sqrt r}\sum_{x\in S_r(\hat x)}\beta(x)
+{1\over \rho}\sum_{x\in S_r(\hat x)}q(x)
\right)\\
&=
O\!\left({N\over r}+{|Q|\over \rho}\right).
\end{align*}
Each leaf region satisfies \eqref{eq:leaf-threshold-rdiv}, and every region has at most \(24\) holes. This proves the theorem.
\end{proof}

\section{A structure theorem for $t$-intersecting curves}\label{sec:intersecting}

The goal of this section is to prove Theorem \ref{main}.  Recall that a set \(\Delta\subset \mathbb R^2\) is \emph{path-connected} if any two
points of \(\Delta\) can be joined by a continuous curve contained in
\(\Delta\).  A \emph{domain} is an open path-connected subset of
\(\mathbb R^2\).  We write \(\partial\Delta\) for the boundary of
\(\Delta\), \(\operatorname{int}(\Delta)\) for its interior, and $\operatorname{cl}(\Delta)$ for its closure.

All path-connected sets \(\Delta\subset\mathbb R^2\) considered in this
paper have boundary contained in a finite arrangement of Jordan arcs and
Jordan curves, where any two arcs or curves meet in only finitely many
points.  Thus the arrangement has finitely many vertices, edges, and faces,
and contains \(\partial\Delta\).  Hence every connected component of
\(\mathbb R^2\setminus\Delta\) is a union of cells of this finite
arrangement, and is therefore path-connected.  We call these components the
\emph{complementary components} of \(\Delta\).

We now establish the following technical result.

\begin{theorem}\label{thm:red-blue-region}
For every \(t\ge 1\), the following holds. Let \(\mathcal B\) be a set of
\(m\) blue curves, and let \(\mathcal G\) be a set of \(n\) green curves, all
contained in a domain \(\Delta\), such that \(\mathcal B\cup\mathcal G\) is
\(t\)-intersecting and in general position.

Then, for every \(\delta\in(0,1)\) and every nonempty point set \(P\)
obtained by selecting at most one endpoint from each curve in \(\mathcal B\),
there is a domain $\Delta'$ with \(\operatorname{cl}(\Delta')\subset\Delta\) such that the following hold.
\begin{enumerate}
    \item $|P\cap\Delta'|\ge \Omega\left({\delta\over t}|P|\right).$

    \item At most \(O(\lceil\sqrt\delta(m+n)\rceil)\) curves in \(\mathcal G\) intersect
    \(\operatorname{cl}(\Delta')\).

    \item The total number of crossing points between the curves in
    \(\mathcal B\cup\mathcal G\) and \(\partial\Delta'\) is
    \(O(\lceil\sqrt\delta(m+n)\rceil)\).

    \item \(\Delta\setminus\operatorname{cl}(\Delta')\) has at most \(24\)
    path-connected components.
\end{enumerate}
\end{theorem}

\begin{proof}
Let $\mathcal B=\{\beta_1,\ldots,\beta_m\}, \mathcal G=\{\gamma_1,\ldots,\gamma_n\}$. We may assume that \(\delta(m+n)^2\ge r_0\), where \(r_0\) is the absolute
constant from Theorem~\ref{thm:mixed}. Otherwise the statement follows by taking \(\Delta'\) to be a sufficiently
small disk around one point of \(P\), chosen in general position.

We construct a plane graph \(G\) as follows. Place a vertex at each
intersection point of two curves in \(\mathcal B\cup\mathcal G\), and also at
the two endpoints of every curve. For each green curve \(\gamma_i\), choose
one additional point \(q_i\in\gamma_i\) that is neither an endpoint nor an
intersection point, and add \(q_i\) as a vertex. Let
\[
Q=\{q_1,\ldots,q_n\}.
\]
The edges of \(G\) are the arcs between consecutive vertices along the curves
in \(\mathcal B\cup\mathcal G\). Thus \(G\) is a plane multigraph contained
in \(\Delta\).   

Since \(\mathcal B\cup\mathcal G\) is \(t\)-intersecting, the number of
intersection points determined by its curves is at most \(t(m+n)^2\). We
refine the plane graph \(G\) inside \(\Delta\) as follows.  Faces and their
sizes are understood in the usual plane embedding of \(G\), with the size of a
face equal to the length of its boundary walk. We eliminate
faces of size \(2\) by, inside $\Delta$, placing one new vertex inside each such face and
joining it to the two boundary vertices. This adds \(O(t(m+n)^2)\) vertices.
We then triangulate all faces of size at least \(4\), again inside
\(\Delta\), and let \(G'\) be the resulting triangulated plane graph. Hence
\[
|V(G')|=O(t(m+n)^2).
\]

Apply Theorem~\ref{thm:mixed} to \(G'\), with the set \(Q\), and with
parameters
\[
r=\delta(m+n)^2,\qquad \rho=\sqrt\delta(m+n).
\]
As a result, we obtain a division \(\mathcal R\) of \(G'\) such that every region
\(R\in\mathcal R\) satisfies
\[
|V(R)|=O(\delta(m+n)^2),\qquad
|b(R)|=O(\sqrt\delta(m+n)),\qquad
|(V(R)\setminus b(R))\cap Q|=O(\sqrt\delta(m+n)),
\]
and has at most \(24\) holes. Moreover,
\[
|\mathcal R|
=
O\left({|V(G')|\over r}+{|Q|\over \rho}\right)
=
O\left({t\over\delta}\right).
\]

For each \(R\in\mathcal R\), let \(\Delta_R\subset \Delta\) be obtained from
the embedded edges of \(R\), together with \(f\cap\Delta\) for every face
\(f\) of \(R\) that is also a face of \(G'\). Since \(R\) is connected,
\(\Delta_R\) is path-connected. Since the regions cover all edges of \(G'\),
the sets \(\Delta_R\) with \(R\in\mathcal R\), cover all vertices of \(G'\), and
hence cover \(P\).  

We next bound the number of path-connected components of
\(\Delta\setminus\Delta_R\). Since \(\Delta_R\) contains the embedded graph
\(R\), each such component is contained in a single face of \(R\). The faces
of \(R\) that are also faces of \(G'\) contribute nothing, by the definition of
\(\Delta_R\). Hence every component of \(\Delta\setminus\Delta_R\) is contained
in \(\Delta\cap h\) for some hole \(h\) of \(R\). Since \(\Delta\) is a domain
and \(\partial h\subset\Delta\), each set \(\Delta\cap h\) is path-connected.
Thus each hole of \(R\) contributes at most one component of
\(\Delta\setminus\Delta_R\). As \(R\) has at most \(24\) holes,
\(\Delta\setminus\Delta_R\) has at most \(24\) path-connected components.

By the pigeonhole principle, there is a region \(R\in\mathcal R\) such that
\[
|P\cap\Delta_R|
\ge { |P| \over |\mathcal R|}
=
\Omega\left({\delta\over t}|P|\right).
\]
Fix this region \(R\).

We next bound the number of green curves that meet \(\Delta_R\). A curve in
\(\mathcal B\cup\mathcal G\) can enter or leave \(\Delta_R\) only through a
boundary vertex of \(R\). By general position, each such boundary vertex lies
on at most two curves. Hence the total number of entry and exit points of all
curves through \(\partial\Delta_R\) is
\[
O(|b(R)|)=O(\sqrt\delta(m+n)).
\]
We claim that
\[
\bigl|\{\gamma_i\in\mathcal G:\gamma_i\cap\Delta_R\ne\emptyset\}\bigr|
\le |(V(R)\setminus b(R))\cap Q|+3|b(R)|.
\]
Indeed, if \(q_i\in V(R)\setminus b(R)\), then \(\gamma_i\) is counted by the
first term. If \(q_i\in b(R)\), then it is counted by the boundary term.
Otherwise \(q_i\notin V(R)\). If \(\gamma_i\) meets \(\Delta_R\), then,
traversing \(\gamma_i\) from \(q_i\) to its first point of entry into
\(\Delta_R\), we meet \(\partial\Delta_R\). Thus \(\gamma_i\) enters
\(\Delta_R\) through a boundary vertex of \(R\). By general position, each
boundary vertex is incident to at most two green curves. This proves the
claim.  By the bounds on \(R\), at most $O(\sqrt\delta(m+n))$ green curves meet \(\Delta_R\).

We may choose a sufficiently small open neighborhood \(\Delta'\) of
\(\Delta_R\), with path-connected closure and
\(\operatorname{cl}(\Delta')\subset\Delta\), chosen so close to \(\Delta_R\)
that \(\Delta\setminus\operatorname{cl}(\Delta')\) has at most \(24\)
path-connected components.  We also choose it small enough so that every green curve
disjoint from \(\Delta_R\) is also disjoint from \(\operatorname{cl}(\Delta')\).
Finally, since $\mathcal B\cup \mathcal G$ is in general position, we can choose \(\partial\Delta'\) so that the number of crossing points
between \(\partial\Delta'\) and the curves in $\mathcal B\cup \mathcal G$ is at most four times the
number of boundary vertices of \(R\). Therefore, the number of crossing points between the curves in $\mathcal B\cup\mathcal G$ and $\partial\Delta'$ is at most $O(\sqrt\delta(m+n))$.  Moreover, 

\[
|P\cap\Delta'|
\ge |P\cap\Delta_R|
\ge \Omega\left({\delta\over t}|P|\right).
\]
Also, by the choice of \(\Delta'\), at most \(O(\sqrt\delta(m+n))\) green
curves intersect \(\operatorname{cl}(\Delta')\).

\end{proof}

Next, we use Theorem~\ref{thm:red-blue-region} to establish a double-grounded-type structure for two collections of curves in the plane.  Given a curve \(\gamma\) in the plane, we arbitrarily designate one endpoint of \(\gamma\) as its first endpoint and the other as its second endpoint.

\medskip

\begin{theorem}\label{thm:doublegrounded}
For every \(t\ge 1\), there exists \(\varepsilon_1=\varepsilon_1(t)>0\) such that the following holds. Let \(\mathcal B\) be a set of \(n\) blue curves and let \(\mathcal G\) be a set of \(n\) green curves in the plane, such that \(\mathcal B\cup\mathcal G\) is $t$-intersecting and in general position.

Then there exist subsets \(\mathcal B'\subseteq \mathcal B\) and \(\mathcal G'\subseteq \mathcal G\), with
\[
|\mathcal B'|,|\mathcal G'|\ge \varepsilon_1 n,
\]
and domains
\[
\Delta_{b_1},\Delta_{b_2},\Delta_{g_1},\Delta_{g_2}\subseteq \mathbb R^2,
\]
such that each curve in \(\mathcal B'\) has its first endpoint in \(\Delta_{b_1}\) and its second endpoint in \(\Delta_{b_2}\), while each curve in \(\mathcal G'\) has its first endpoint in \(\Delta_{g_1}\) and its second endpoint in \(\Delta_{g_2}\). Moreover, every curve in \(\mathcal B'\) avoids \(\Delta_{g_1}\cup\Delta_{g_2}\), every curve in \(\mathcal G'\) avoids \(\Delta_{b_1}\cup\Delta_{b_2}\), and
\[
(\Delta_{b_1}\cup\Delta_{b_2})\cap(\Delta_{g_1}\cup\Delta_{g_2})=\emptyset .
\]
\end{theorem}

\begin{proof}
Let \(\delta_1=\delta_1(t)>0\) be a sufficiently small constant to be chosen later.
All implicit constants in this proof are absolute or depend only on \(t\). We
choose all domains in general position, so that their boundaries avoid all
curve endpoints and all curves properly cross the boundaries.

We start by applying Theorem~\ref{thm:red-blue-region} in the domain
\(\mathbb R^2\), with parameter \(\delta_1\), blue family \(\mathcal B\),
green family \(\mathcal G\), and \(P\) equal to the set of first endpoints
of the curves in \(\mathcal B\).  We obtain a domain \(\Delta_{b_1}\) such that at least
\(c_1n\) curves in \(\mathcal B\) have their first endpoint in
\(\Delta_{b_1}\), where \(c_1=c_1(t)>0\), and such that at most
\(O(\sqrt\delta_1 n)\) curves in \(\mathcal G\) meet
\(\operatorname{cl}(\Delta_{b_1})\). Let \(\mathcal B_1\) be the set of blue
curves whose first endpoint lies in \(\Delta_{b_1}\), and let
\(\mathcal G_0\) be the set of green curves disjoint from
\(\operatorname{cl}(\Delta_{b_1})\). Choosing \(\delta_1>0\) sufficiently
small, we have
\[
|\mathcal B_1|\ge c_1n,\qquad |\mathcal G_0|\ge n/2.
\]

We now find a domain for the second endpoints of many curves in
\(\mathcal B_1\). There are two cases.

\medskip
\noindent\emph{Case 1.}
Suppose that at least \(|\mathcal B_1|/2\) curves in \(\mathcal B_1\) have
their second endpoint in \(\Delta_{b_1}\). Let \(\mathcal B_2\) be this
subfamily, set $\Delta_{b_2}=\Delta_{b_1},$ and let \(\mathcal G_2=\mathcal G_0\). Then every curve in \(\mathcal B_2\)
has its first endpoint in \(\Delta_{b_1}\) and its second endpoint in
\(\Delta_{b_2}\), every curve in \(\mathcal G_2\) avoids
\(\Delta_{b_1}\cup\Delta_{b_2}\), and
\[
|\mathcal B_2|\ge {c_1\over 2}n,\qquad |\mathcal G_2|\ge n/2.
\]

\medskip
\noindent\emph{Case 2.}
Suppose that at least \(|\mathcal B_1|/2\) curves in \(\mathcal B_1\) have
their second endpoint outside \(\operatorname{cl}(\Delta_{b_1})\). By
Theorem~\ref{thm:red-blue-region}, $\mathbb R^2\setminus\operatorname{cl}(\Delta_{b_1})$ 
has at most \(24\) path-connected components. Hence one of these components,
call it \(\Delta^{(1)}\), contains the second endpoints of at least
\(|\mathcal B_1|/48\) curves in \(\mathcal B_1\). Let
\(\mathcal B_1^*\) be this subfamily, and let \(Q_{b_2}\) be the set of
their second endpoints.

Let \(\widetilde{\mathcal B}_1\) be the family of subarcs obtained by cutting
the curves in \(\mathcal B_1\) at their intersection points with
\(\partial\Delta^{(1)}\), and keeping the subarcs that lie in
\(\Delta^{(1)}\). If necessary, shorten these subarcs slightly so that they
are contained in \(\Delta^{(1)}\). Since $\partial\Delta^{(1)}\subseteq\partial\Delta_{b_1}$ and Theorem~\ref{thm:red-blue-region} gives \(O(\sqrt\delta_1 n)\) crossings
between the curves in \(\mathcal B\cup\mathcal G\) and \(\partial\Delta_{b_1}\), we have $
|\widetilde{\mathcal B}_1|=O(n).$  Moreover, each point of \(Q_{b_2}\) is an endpoint of one curve in
\(\widetilde{\mathcal B}_1\). By general position, after an arbitrarily small
perturbation if necessary, \(\widetilde{\mathcal B}_1\) together with the
green curves contained in \(\Delta^{(1)}\) is still \(t\)-intersecting and in
general position.

Let \(\mathcal G^{(1)}\subseteq\mathcal G_0\) be the set of green curves
contained in \(\Delta^{(1)}\). Apply Theorem~\ref{thm:red-blue-region} in the domain
\(\Delta^{(1)}\), with parameter \(\delta_1\), blue family
\(\widetilde{\mathcal B}_1\), green family \(\mathcal G^{(1)}\), and
\(P=Q_{b_2}\). We obtain a domain \(\Delta_{b_2}\subseteq\Delta^{(1)}\) such
that at least \(c_2n\) curves in \(\mathcal B_1^*\) have their second
endpoint in \(\Delta_{b_2}\), where \(c_2=c_2(t)>0\), and such that at most
\[
O\bigl(\sqrt\delta_1(|\widetilde{\mathcal B}_1|+|\mathcal G^{(1)}|)\bigr)
=O(\sqrt\delta_1 n)
\]
curves in \(\mathcal G^{(1)}\) meet \(\operatorname{cl}(\Delta_{b_2})\).
Let \(\mathcal B_2\subseteq\mathcal B_1^*\) be the corresponding blue
subfamily. Let \(\mathcal G_2\) be the set of curves in \(\mathcal G_0\)
disjoint from \(\operatorname{cl}(\Delta_{b_2})\). The only curves of
\(\mathcal G_0\) that can meet \(\operatorname{cl}(\Delta_{b_2})\) are those
contained in \(\Delta^{(1)}\), since
\(\operatorname{cl}(\Delta_{b_2})\subset\Delta^{(1)}\). Hence, choosing
\(\delta_1>0\) sufficiently small,
\[
|\mathcal B_2|\ge c_2n,\qquad |\mathcal G_2|\ge n/3.
\]
Every curve in \(\mathcal B_2\) has its first endpoint in \(\Delta_{b_1}\)
and its second endpoint in \(\Delta_{b_2}\), and every curve in
\(\mathcal G_2\) avoids \(\Delta_{b_1}\cup\Delta_{b_2}\).

\medskip
In either case, for a constant $c_3  = c_3(t)$ sufficiently small, we have obtained
subfamilies \(\mathcal B_2\subseteq\mathcal B\) and
\(\mathcal G_2\subseteq\mathcal G\), with $|\mathcal B_2|,|\mathcal G_2|\ge c_3n,$ 
and domains \(\Delta_{b_1},\Delta_{b_2}\), such that every curve in
\(\mathcal B_2\) has its first endpoint in \(\Delta_{b_1}\) and its second
endpoint in \(\Delta_{b_2}\), and every curve in \(\mathcal G_2\) avoids
\(\Delta_{b_1}\cup\Delta_{b_2}\).

We now pass to one component containing many green curves. In Case 1,
\[
\mathbb R^2\setminus\operatorname{cl}(\Delta_{b_1})
=
\mathbb R^2\setminus\operatorname{cl}(\Delta_{b_1}\cup\Delta_{b_2})
\]
has at most \(24\) path-connected components. In Case 2,
\(\mathbb R^2\setminus\operatorname{cl}(\Delta_{b_1})\) has at most \(24\)
components, and Theorem~\ref{thm:red-blue-region}, applied inside
\(\Delta^{(1)}\), gives that $\Delta^{(1)}\setminus\operatorname{cl}(\Delta_{b_2})$ has at most \(24\) path-connected components. Therefore, in either case, $\mathbb R^2\setminus
\operatorname{cl}(\Delta_{b_1}\cup\Delta_{b_2})$ has at most 48 path-connected components. Since
every curve in \(\mathcal G_2\) is connected and avoids
\(\operatorname{cl}(\Delta_{b_1}\cup\Delta_{b_2})\), one such component
contains a constant fraction of the curves in \(\mathcal G_2\). Let this
component be \(\Delta^{(2)}\), and replace \(\mathcal G_2\) by the subfamily
contained in \(\Delta^{(2)}\). Since \(c_3\) is sufficiently small, we still have $
|\mathcal B_2|,|\mathcal G_2|\ge c_3n,$ every curve in \(\mathcal G_2\) is contained in the domain \(\Delta^{(2)}\),
and
\[
\Delta^{(2)}\cap(\Delta_{b_1}\cup\Delta_{b_2})=\emptyset.
\]
Moreover, $\partial\Delta^{(2)}
\subseteq
\partial\Delta_{b_1}\cup\partial\Delta_{b_2}.$

We now work inside \(\Delta^{(2)}\) to ground the endpoints of many green
curves. At this point, the constant \(c_3=c_3(t)>0\) is fixed. Let \(\delta_2=\delta_2(t)>0\) be sufficiently small compared with \(c_3\).   Let \(\widetilde{\mathcal B}_2\) be the family of subarcs obtained by
cutting the curves in \(\mathcal B_2\) at their intersection points with
\(\partial\Delta^{(2)}\), and keeping the subarcs that lie in
\(\Delta^{(2)}\). If necessary, shorten these subarcs slightly so that they
are contained in \(\Delta^{(2)}\). By the
at most two applications of Theorem~\ref{thm:red-blue-region} used to choose
\(\Delta_{b_1}\) and \(\Delta_{b_2}\), the curves of
\(\mathcal B_2\) cross \(\partial\Delta^{(2)}\) only \(O(n)\) times in total.
Thus $|\widetilde{\mathcal B}_2|=O(n).$  After an arbitrarily small perturbation if necessary,
the family \(\mathcal G_2\cup\widetilde{\mathcal B}_2\) is still
\(t\)-intersecting and in general position.

Apply Theorem~\ref{thm:red-blue-region} in the domain
\(\Delta^{(2)}\), with parameter \(\delta_2\), blue family
\(\mathcal G_2\), green family \(\widetilde{\mathcal B}_2\), and \(P\)
equal to the set of first endpoints of the curves in \(\mathcal G_2\). We obtain a constant $c_4 = c_4(t)$ and a domain
\(\Delta_{g_1}\subseteq\Delta^{(2)}\) such that at least \(c_4n\) curves in
\(\mathcal G_2\) have their first endpoint in \(\Delta_{g_1}\), and such
that at most
\[
O\bigl(\sqrt\delta_2(|\mathcal G_2|+|\widetilde{\mathcal B}_2|)\bigr)
=O(\sqrt\delta_2 n)
\]
curves in \(\widetilde{\mathcal B}_2\) meet
\(\operatorname{cl}(\Delta_{g_1})\). Let \(\mathcal G_3\subseteq
\mathcal G_2\) be the corresponding green subfamily. Let
\(\mathcal B_3\subseteq\mathcal B_2\) be the set of blue curves none of whose
subarcs in \(\widetilde{\mathcal B}_2\) meets
\(\operatorname{cl}(\Delta_{g_1})\). Each discarded blue curve accounts for
at least one discarded subarc, so we discard only \(O(\sqrt\delta_2 n)\) blue
curves. Taking \(\delta_2\) sufficiently small compared with \(c_3\), we have
\[
|\mathcal G_3|\ge c_4n,\qquad |\mathcal B_3|\ge {c_3\over 2}n.
\]
Moreover, every curve in \(\mathcal B_3\) avoids \(\Delta_{g_1}\).

Finally, let \(\widetilde{\mathcal B}_3\) be the family of subarcs obtained
by cutting the curves in \(\mathcal B_3\) at their intersection points with
\(\partial\Delta^{(2)}\), and keeping the subarcs that lie in
\(\Delta^{(2)}\). As above, \(|\widetilde{\mathcal B}_3|=O(n)\), and
\(\mathcal G_3\cup\widetilde{\mathcal B}_3\) may be assumed to be
\(t\)-intersecting and in general position.

Apply Theorem~\ref{thm:red-blue-region} in the domain
\(\Delta^{(2)}\), with parameter \(\delta_2\), blue family
\(\mathcal G_3\), green family \(\widetilde{\mathcal B}_3\), and \(P\)
equal to the set of second endpoints of the curves in \(\mathcal G_3\). We obtain a domain
\(\Delta_{g_2}\subseteq\Delta^{(2)}\) such that at least \(c_5n\) curves in
\(\mathcal G_3\) have their second endpoint in \(\Delta_{g_2}\), and such
that at most
\[
O\bigl(\sqrt\delta_2(|\mathcal G_3|+|\widetilde{\mathcal B}_3|)\bigr)
=O(\sqrt\delta_2 n)
\]
curves in \(\widetilde{\mathcal B}_3\) meet
\(\operatorname{cl}(\Delta_{g_2})\). Let \(\mathcal G'\subseteq
\mathcal G_3\) be the corresponding green subfamily. Let
\(\mathcal B'\subseteq\mathcal B_3\) be the set of blue curves none of whose
subarcs in \(\widetilde{\mathcal B}_3\) meets
\(\operatorname{cl}(\Delta_{g_2})\). Again, taking \(\delta_2\) sufficiently small compared with \(c_3\), we have
\[
|\mathcal G'|\ge c_5n,\qquad |\mathcal B'|\ge c_6n
\]
for some \(c_5 = c_5(t), c_6=c_6(t)>0\).

Now every curve in \(\mathcal B'\subseteq\mathcal B_2\) has its first
endpoint in \(\Delta_{b_1}\) and its second endpoint in \(\Delta_{b_2}\).
Every curve in \(\mathcal G'\) has its first endpoint in \(\Delta_{g_1}\)
and its second endpoint in \(\Delta_{g_2}\). Since
\[
\Delta_{g_1},\Delta_{g_2}\subseteq\Delta^{(2)}
\qquad\text{and}\qquad
\Delta^{(2)}\cap(\Delta_{b_1}\cup\Delta_{b_2})=\emptyset,
\]
we have
\[
(\Delta_{b_1}\cup\Delta_{b_2})\cap
(\Delta_{g_1}\cup\Delta_{g_2})=\emptyset.
\]
Also, every curve in \(\mathcal G'\subseteq\mathcal G_2\) is contained in
\(\Delta^{(2)}\), and hence avoids \(\Delta_{b_1}\cup\Delta_{b_2}\). By
construction, every curve in \(\mathcal B'\) avoids both
\(\Delta_{g_1}\) and \(\Delta_{g_2}\). Setting $\varepsilon_1=\min\{c_5,c_6\},$ we get $|\mathcal B'|,|\mathcal G'|\ge \varepsilon_1 n.$
\end{proof}

We now prove \Cref{main} when the two families have the same size.

\begin{theorem}\label{mainbalance}
For every integer \(t \ge 1\), there exists a constant \(\varepsilon_2>0\) with
the following property. Let \(\mathcal B\) be a set of \(n\) blue curves, and
let \(\mathcal G\) be a set of \(n\) green curves in the plane such that
\(\mathcal B\cup \mathcal G\) is a collection of \(t\)-intersecting curves in
general position. Then there exist subsets
\(\mathcal B'\subseteq \mathcal B\) and
\(\mathcal G'\subseteq \mathcal G\), with
\[
|\mathcal B'|,|\mathcal G'|\ge \varepsilon_2 n,
\]
such that either every curve in \(\mathcal B'\) intersects every curve in
\(\mathcal G'\) an odd number of times, or every such pair intersects an even
number of times.
\end{theorem}

\begin{proof}

Let $\mathcal B=\{\beta_1,\ldots,\beta_n\}, \mathcal G=\{\gamma_1,\ldots,\gamma_n\}.$   We start by applying Theorem \ref{thm:doublegrounded} to $\mathcal B$ and $\mathcal G$, and obtain subsets \(\mathcal{B}'\subseteq \mathcal{B}\) and \(\mathcal{G}'\subseteq \mathcal{G}\), together with domains
\[
\Delta_{b_1},\Delta_{b_2},\Delta_{g_1},\Delta_{g_2}\subseteq \mathbb{R}^2,
\]
such that
\[
|\mathcal{B}'|\ge \varepsilon_1|\mathcal{B}|
\qquad\text{and}\qquad
|\mathcal{G}'|\ge \varepsilon_1|\mathcal{G}|,
\]
where $\varepsilon_1 = \varepsilon_1(t).$  Moreover, each curve in \(\mathcal{B}'\) has its first endpoint in \(\Delta_{b_1}\) and its second endpoint in \(\Delta_{b_2}\), while each curve in \(\mathcal{G}'\) has its first endpoint in \(\Delta_{g_1}\) and its second endpoint in \(\Delta_{g_2}\). Furthermore, every blue curve in \(\mathcal{B}'\) avoids \(\Delta_{g_1}\cup\Delta_{g_2}\), every green curve in \(\mathcal{G}'\) avoids \(\Delta_{b_1}\cup\Delta_{b_2}\), and $(\Delta_{b_1}\cup\Delta_{b_2}) \cap (\Delta_{g_1}\cup\Delta_{g_2}) = \emptyset$.

For \(\beta\in\mathcal B'\) and \(\gamma\in\mathcal G'\), let \(\chi(\beta,\gamma)\) be the number of intersections between \(\beta\) and \(\gamma\).  Clearly, all such intersection points lie outside of $\Delta_{b_1}\cup\Delta_{b_2} \cup \Delta_{g_1}\cup\Delta_{g_2}$.

\medskip

\begin{claim}
For any \(\beta,\beta'\in\mathcal B'\) and \(\gamma,\gamma'\in\mathcal G'\),

\[\chi(\beta,\gamma)+\chi(\beta',\gamma)+\chi(\beta,\gamma')+\chi(\beta',\gamma') \equiv 0 \pmod 2.\]
\end{claim}

\begin{proof}
In the following proof, we may assume that $\beta\neq \beta'$ and $\gamma \neq \gamma'$ since otherwise the claim is immediate. Given $\beta$ and $\beta'$, we create a closed curve $\beta''$ as follows.  Since both first endpoints of $\beta$ and $\beta'$ lie in $\Delta_{b_1}$, and $\Delta_{b_1}$ is a domain, there is an arc that connects these endpoints that lies in $\Delta_{b_1}$.  Likewise, since both second endpoints of \(\beta\) and \(\beta'\) lie in \(\Delta_{b_2}\), and \(\Delta_{b_2}\) is a domain, there is an arc joining these endpoints inside \(\Delta_{b_2}\).   As a result, we obtain a closed curve $\beta''$ in the plane, not necessarily simple.  Since $\Delta_{b_1}$ and $\Delta_{b_2}$ are domains, we can assume that $\beta''$ has a finite number of self-intersections, and no self tangencies.

We repeat the same procedure above for $\gamma$ and $\gamma'$, to obtain a
closed curve $\gamma''$. Again, we can assume that $\gamma''$ has a finite
number of self-intersections, and no self tangencies. Since every curve in
\(\mathcal B'\) avoids \(\Delta_{g_1}\cup\Delta_{g_2}\), every curve in
\(\mathcal G'\) avoids \(\Delta_{b_1}\cup\Delta_{b_2}\), and
\[
(\Delta_{b_1}\cup\Delta_{b_2})\cap
(\Delta_{g_1}\cup\Delta_{g_2})=\emptyset,
\]
the crossings between \(\beta''\) and \(\gamma''\) are exactly the crossings
between one of \(\beta,\beta'\) and one of \(\gamma,\gamma'\). Therefore
\[
|\beta''\cap\gamma''|
=
\chi(\beta,\gamma)+\chi(\beta',\gamma)+\chi(\beta,\gamma')
+\chi(\beta',\gamma').
\]
Since two closed curves in the plane properly cross an even number of times,
the summation above is even.\end{proof}

Now fix \(\beta_0\in\mathcal B'\) and \(\gamma_0\in\mathcal G'\). Define functions \(f_1:\mathcal B'\to  \{0,1\}\) and \(f_2:\mathcal G'\to\{0,1\}\) by
\[
f_1(\beta)\equiv \chi(\beta,\gamma_0)\pmod 2
\]
and
\[
f_2(\gamma)\equiv \chi(\beta_0,\gamma_0)+\chi(\beta_0,\gamma)\pmod 2.
\]
By the claim, we know that 
\[
\chi(\beta,\gamma) \equiv f_1(\beta)+f_2(\gamma) \pmod 2. 
\]
Partition 
\[
\mathcal{B}'=\mathcal{B}'_0 \cup \mathcal{B}'_1
\qquad\text{and}\qquad
\mathcal{G}'=\mathcal{G}'_0 \cup \mathcal{G}'_1
\]
where 
\[
\mathcal{B}'_i:=\{\beta \in \mathcal{B}': f_1(\beta)\equiv i (\text{mod } 2)\}
\qquad\text{and}\qquad
\mathcal{G}'_j:=\{\gamma \in \mathcal{G}': f_2(\gamma)\equiv j (\text{mod } 2)\}. 
\]

Thus, for \(\beta\in\mathcal B'_i\) and \(\gamma\in\mathcal G'_j\), the parity of \(\chi(\beta,\gamma)\) is \(i+j\pmod 2\).   By the pigeonhole principle, for some pair \(i,j\in\{0,1\}\), we have $|\mathcal{B}'_i|=\Omega(\varepsilon_1|\mathcal B|)$ and $|\mathcal{G}'_j| =\Omega (\varepsilon_1|\mathcal G|)$, and either every blue curve in $\mathcal{B}'_i$ crosses every green curve in $\mathcal{G}'_j$ an odd number of times, or every blue curve in $\mathcal{B}'_i$ crosses every green curve in $\mathcal{G}'_j$ an even number of times.  By setting $\varepsilon_2 = \Theta(\varepsilon_1)$ sufficiently small, the statement follows.\end{proof}

\noindent We now prove Theorem \ref{main}.

\begin{proof}[Proof of Theorem \ref{main}]

Without loss of generality, we can assume that $|\mathcal B| \leq |\mathcal G|$.  Let \(\mathcal B\) be a family of \(m\) blue curves and let
\(\mathcal G\) be a family of \(n\) green curves, where \(m\leq n\), such that
\(\mathcal B\cup\mathcal G\) is a collection of \(t\)-intersecting curves in
general position. We show that the conclusion holds with
\(\varepsilon=\varepsilon_2/2\).

For each \(\beta\in\mathcal B\), replace \(\beta\) by either $\left\lfloor \frac{n}{m}\right\rfloor$ or $\left\lceil \frac{n}{m}\right\rceil$ copies drawn in a sufficiently small neighborhood of \(\beta\), choosing the numbers of copies
so that the resulting blue family \(\widetilde{\mathcal B}\) has exactly
\(n\) members. The copies may be chosen so that
\(\widetilde{\mathcal B}\cup\mathcal G\) is still a collection of
\(t\)-intersecting curves in general position and every copy of \(\beta\)
intersects each \(\gamma\in\mathcal G\) with the same number of times as \(\beta\).

Apply Theorem~\ref{mainbalance} to
\(\widetilde{\mathcal B}\) and \(\mathcal G\). We obtain subfamilies
\(\widetilde{\mathcal B}'\subseteq\widetilde{\mathcal B}\) and
\(\mathcal G'\subseteq\mathcal G\) such that
\[
|\widetilde{\mathcal B}'|,|\mathcal G'|\ge \varepsilon_2 n,
\]
and all pairs in
\(\widetilde{\mathcal B}'\times\mathcal G'\) have the same intersection
parity.

Let \(\mathcal B'\subseteq\mathcal B\) be the set of original blue curves
having at least one copy in \(\widetilde{\mathcal B}'\). Since each original
blue curve has at most \(\lceil n/m\rceil\) copies,
\[
|\mathcal B'|
\ge
\frac{|\widetilde{\mathcal B}'|}{\lceil n/m\rceil}
\ge
\frac{\varepsilon_2 n}{\lceil n/m\rceil} \geq \frac{\varepsilon_2}{2}m.
\]
Moreover, $|\mathcal G'|\ge  \frac{\varepsilon_2}{2}n.$  By the construction of the copies, either every pair in
\(\mathcal B'\times\mathcal G'\) intersects an even number of times, or every
such pair intersects an odd number of times.    
\end{proof}

\section{Topological graphs with no $k$-pairwise odd crossing edges}\label{sec:top}

In this section, we use Theorem \ref{thm:dense} to prove Theorem \ref{oddcross}.  For a graph \(G=(V,E)\), the \emph{bisection width} \(b(G)\) is the minimum number of edges whose removal partitions \(V\) into two disjoint sets \(V_1\cup V_2\) such that
\[
\frac{1}{3}|V|\le |V_1|,|V_2|\le \frac{2}{3}|V|.
\]
Equivalently, \(b(G)\) is the minimum number of edges crossing between two parts of such a balanced partition.

The \emph{odd crossing number} of a graph \(G\), denoted by \(\operatorname{odd\text{-}cr}(G)\), is the minimum number of pairs of edges that cross an odd number of times, taken over all drawings of \(G\) in the plane.  We will need the following result due to Pach and T\'oth.

\begin{lemma}[\cite{PTdisjoint}]\label{lem:mat}
There is an absolute constant \(c_1>0\) such that every graph \(G\) with \(n\) vertices satisfies
\[
b(G)\le c_1\log n \sqrt{\operatorname{odd\text{-}cr}(G)+\sum_{v\in V(G)}d(v)^2},
\]
where \(d(v)\) denotes the degree of the vertex \(v\).
\end{lemma}

\begin{proof}[Proof of Theorem \ref{oddcross}]
Let \(f(n,k)\) denote the maximum number of edges in an \(n\)-vertex topological graph with no \(k\) edges that pairwise cross an odd number of times, and whose edge set is \(t\)-intersecting. We prove by double induction on \(n\) and \(k\) that
\begin{equation}\label{fnk}
    f(n,k)\le n(\log n)^{C\log k},
\end{equation}
where \(C=C(t)\) is a sufficiently large constant.

Let $n_0$ be a large constant that will be determined later.  The base case \(n\leq n_0\) is trivial by setting $C$ sufficiently large. The base case \(k=2\) and $n \geq 3$ follows from the Hanani--Tutte theorem by setting  \(C\) sufficiently large. Now assume that \eqref{fnk} holds for  \(n'<n\) or \(k'<k\).

Let \(G\) be an \(n\)-vertex topological graph with no \(k\) edges that
pairwise cross an odd number of times, and whose edges are \(t\)-intersecting.
After shortening each edge in a sufficiently small neighborhood of its
endpoints, we regard the edges as a family \(E\) of Jordan arcs in general
position. Clearly, this can be done without changing the proper crossings
between edges. We consider two cases.

\medskip

\noindent \emph{Case 1.} Suppose that
\[
\operatorname{odd\text{-}cr}(G)\le \frac{|E|^2}{\log^6 n}.
\]
By Lemma~\ref{lem:mat}, there is a constant $c_1 > 0$ such that
\[
b(G)\le c_1\log n\sqrt{\operatorname{odd\text{-}cr}(G)+\sum_{v\in V(G)}d(v)^2}.
\]
Since we always have
\[
\sum_{v\in V(G)}d(v)^2\le n\sum_{v\in V(G)}d(v)=2n|E|,
\]
it follows that
\[
b(G)\le c_1\log n\sqrt{\frac{|E|^2}{\log^6 n}+2n|E|}.
\]
If $|E|^2/\log^6 n<2n|E|$, then we have $|E|<2n\log^6 n$ and we are done since $C$ is large.   Hence we may assume that $|E|^2/\log^6 n \geq2n|E|$, which implies
 
\[
b(G)\le \sqrt{2}c_1\frac{|E|}{\log^2 n}.
\]

Let \(V(G)=V_1\cup V_2\) be a partition with at most \(b(G)\) edges between \(V_1\) and \(V_2\), where \(|V_1|,|V_2|\le 2n/3\). Since each induced subgraph \(G[V_i]\) also contains no \(k\) edges that pairwise cross an odd number of times, the induction hypothesis on \(n\) gives
\begin{align*}
|E|
    &\le f(|V_1|,k)+f(|V_2|,k)+b(G) \\
    &\le |V_1|(\log |V_1|)^{C\log k}
        + |V_2|(\log |V_2|)^{C\log k}
        + b(G) \\
    &\le |V_1|(\log (2n/3))^{C\log k}
        + |V_2|(\log (2n/3))^{C\log k}
        + b(G) \\
    &= n\bigl(\log n-\log(3/2)\bigr)^{C\log k}+b(G) \\
    &\le n(\log n)^{C\log k}-n(\log n)^{C\log k-1}+b(G) \\
    &\le n(\log n)^{C\log k}-n(\log n)^{C\log k-1}
        +\sqrt{2}c_1\frac{|E|}{\log^2 n}.
\end{align*}

\noindent Rearranging the last inequality gives
\[
\left(1-\frac{\sqrt{2}c_1}{\log^2 n}\right)|E|
\le n(\log n)^{C\log k}-n(\log n)^{C\log k-1}.
\]
For \(n_0\) sufficiently large,
\[
\left(1-\frac{\sqrt{2}c_1}{\log^2 n}\right)^{-1}
\le 1+\frac{4c_1}{\log^2 n}.
\]
Hence
\[
\begin{aligned}
|E|
&\le
\left(n(\log n)^{C\log k}-n(\log n)^{C\log k-1}\right)
\left(1+\frac{4c_1}{\log^2 n}\right)\\
&\le n(\log n)^{C\log k}
-n(\log n)^{C\log k-1}
+4c_1\,n(\log n)^{C\log k-2}\\
&\le n(\log n)^{C\log k},
\end{aligned}
\]
where the last inequality holds for \(n_0\) and $C$ sufficiently large.

\medskip
\noindent \emph{Case 2.} Suppose that
\[
\operatorname{odd\text{-}cr}(G)> \frac{|E|^2}{\log^6 n}.
\]

Let \(E=E_1\cup E_2\) be a uniformly random equipartition, so that
\[
|E_1|=\lfloor |E|/2\rfloor,\qquad |E_2|=\lceil |E|/2\rceil .
\]
For every odd-crossing pair of edges, the probability that the two edges lie
in different parts is at least \(1/2\). Hence the expected number of
odd-crossing pairs with one edge in \(E_1\) and the other in \(E_2\) is at least
\[
\frac{1}{2}\operatorname{odd\text{-}cr}(G)
> \frac{|E|^2}{2\log^6 n}.
\]
Therefore there is an equipartition \(E=E_1\cup E_2\) for which the number of
such pairs is at least \(|E|^2/(2\log^6 n)\). Since $|E_1||E_2|\le |E|^2/4$, the density of odd-crossing pairs between \(E_1\) and \(E_2\) is at least 
\[
\frac{|E|^2/(2\log^6 n)}{|E_1||E_2|}
\ge \frac{1}{\log^6 n}.
\]
By Theorem~\ref{thm:dense}, there is a constant \(c=c(t)\) and subsets
\[
E_1'\subseteq E_1,\qquad E_2'\subseteq E_2
\]
such that every edge in \(E_1'\) crosses every edge in \(E_2'\) an odd number
of times, and
\[
|E_i'|\ge \frac{|E_i|}{(\log^6 n)^c}
\qquad\text{for } i=1,2.
\]
In particular, since \(E_1,E_2\) are an equipartition,
\[
|E_1'|,|E_2'|\ge \frac{|E|}{3(\log^6 n)^c}.
\]

If both \(E_1'\) and \(E_2'\) contain \(\lceil k/2\rceil\) edges that pairwise
cross an odd number of times, then, together with the fact that every edge in
\(E_1'\) crosses every edge in \(E_2'\) an odd number of times, we obtain at
least \(k\) edges that pairwise cross an odd number of times, a contradiction.
Hence one of \(E_1'\) or \(E_2'\) contains no \(\lceil k/2\rceil\) edges that
pairwise cross an odd number of times. Without loss of generality, assume this
holds for \(E_1'\).

The topological graph with vertex set \(V(G)\) and edge set \(E_1'\) has
\(t\)-intersecting edge set and contains no \(\lceil k/2\rceil\) edges that
pairwise cross an odd number of times. By the induction hypothesis on \(k\),
\[
|E_1'|
\le f(n,\lceil k/2\rceil)
\le n(\log n)^{C\log \lceil k/2\rceil}.
\]
Therefore
\[
\frac{|E|}{3(\log^6 n)^c}
\le n(\log n)^{C\log \lceil k/2\rceil}.
\]
Since \(k\ge 3\), we have \(\lceil k/2\rceil\le 2k/3\). Hence
\[
\log \lceil k/2\rceil
\le \log k-\log(3/2),
\]
and so
\[
|E|
\le 3n(\log n)^{C\log k-C\log(3/2)+6c}.
\]
Choosing \(C\) sufficiently large compared to \(c\), and then choosing
\(n_0\) sufficiently large, we obtain
\[
|E|\le n(\log n)^{C\log k}.
\]

 \end{proof}

\section{Concluding remarks}

Although Theorem~\ref{mainbalance} is stated for \(t\)-intersecting curves,
its proof only uses that the total number of crossing points in the
arrangement is \(O(n^2)\). Following the proof of
Theorem~\ref{mainbalance} almost verbatim, we obtain the following.

 \begin{theorem}
For every constant \(C>0\), there exists a constant
\(\varepsilon=\varepsilon(C)>0\) with the following property. Let
\(\mathcal B\) be a set of \(n\) blue curves, and let \(\mathcal G\) be a set
of \(n\) green curves in the plane such that
\(\mathcal B\cup\mathcal G\) is in general position and determines at most
\(Cn^2\) crossing points. Then there exist subsets
\(\mathcal B'\subseteq\mathcal B\) and
\(\mathcal G'\subseteq\mathcal G\), with
\[
|\mathcal B'|,|\mathcal G'|\ge \varepsilon n,
\]
such that either every curve in \(\mathcal B'\) intersects every curve in
\(\mathcal G'\) an odd number of times, or every such pair intersects an even
number of times.
\end{theorem}

As mentioned in the introduction, applying the weak bipartite regularity
lemma to Theorem~\ref{main} also gives the following density result for even
crossings.

\begin{theorem}\label{thm:dense-even}
For every integer \(t\ge 1\), there exists a constant \(c=c(t)>0\) with the
following property. Let \(\mathcal B\) be a set of blue curves, and let
\(\mathcal G\) be a set of green curves in the plane such that
\(\mathcal B\cup\mathcal G\) is a collection of \(t\)-intersecting curves in
general position. Suppose that at least
\(\delta|\mathcal B||\mathcal G|\) pairs
\((\beta,\gamma)\in\mathcal B\times\mathcal G\) intersect an even number of
times. Then there exist subsets
\(\mathcal B'\subseteq\mathcal B\) and
\(\mathcal G'\subseteq\mathcal G\), with
\[
|\mathcal B'|\ge \delta^c|\mathcal B|
\qquad\text{and}\qquad
|\mathcal G'|\ge \delta^c|\mathcal G|,
\]
such that every curve in \(\mathcal B'\) intersects every curve in
\(\mathcal G'\) an even number of times.
\end{theorem}

Let us remark that Theorem~\ref{thm:mixed} can be extended to also bound $|V(R)\cap Q|=O(\rho)$ and $|b(R)\cap Q|=O(\sqrt{\rho})$ while leaving everything else unchanged. One can also simultaneously control multiple prescribed vertex sets and weight functions for each region in the final division. However, this requires a more involved argument, which will appear in a forthcoming paper.

\end{document}